\documentclass[11pt,a4paper]{article}

\usepackage{geometry}
\usepackage{amsmath,amssymb,amsfonts,amsthm}
\usepackage{latexsym}
\usepackage{indentfirst}
\usepackage{graphicx}
\usepackage{mathrsfs}
\usepackage{xcolor}
\usepackage[colorlinks=true,linkcolor=blue,citecolor=blue,urlcolor=blue]{hyperref}

\numberwithin{equation}{section}

\theoremstyle{plain}
\newtheorem{Thm}{Theorem}[section]
\newtheorem{Lem}[Thm]{Lemma}

\theoremstyle{definition}
\newtheorem{Def}[Thm]{Definition}

\theoremstyle{remark}

\newcommand{\N}{\mathbb{N}}

\newcommand{\R}{\mathbb{R}}

\begin{document}

\title{Infinitely many sign-changing solutions for logarithmic Schr\"odinger equations via an \(L^p\)-perturbation approach}

\author{Chen Huang, Zhipeng Yang\thanks{Corresponding author: yangzhipeng326@163.com}, and Jiazheng Zhou}

\date{}

\AtEndDocument{%
  \par
  \bigskip

 \noindent
  \textbf{Chen Huang}\\[0.2em]
  \textsc{School of Mathematics, University of Shanghai for Science and Technology, Shanghai, China}\\
  \textsc{Yunnan Key Laboratory of Modern Analytical Mathematics and Applications, Kunming, China}\\[0.3em]
  \textit{E-mail address}: \texttt{chenhuangmath111@163.com}%
  
   \bigskip
  \noindent
  \textbf{Zhipeng Yang}\\[0.2em]
  \textsc{Department of Mathematics, Yunnan Normal University, Kunming, China}\\
  \textsc{Yunnan Key Laboratory of Modern Analytical Mathematics and Applications, Kunming, China}\\[0.3em]
  \textit{E-mail address}: \texttt{yangzhipeng326@163.com}%
  
 \bigskip
   \noindent
  \textbf{Jiazheng Zhou}\\[0.2em]
  \textsc{Departamento de Matemática, Universidade de Brasília, Brasília, Brazil}\\
  \textit{E-mail address}: \texttt{zhou@mat.unb.br}\\[1.5em]
}
\date{}
\maketitle
\begin{abstract}
We study the logarithmic Schr\"odinger equation
\[
-\Delta u+V(x)u=u\log u^2,\qquad x\in\R^N,\ N\ge3.
\]
Since the logarithmic energy is not \(C^1\) on the natural space \(H_V^1(\R^N)\), direct
invariant-set minimax arguments for sign-changing solutions are not available. We introduce an
\(L^p\)-regularization perturbation, which restores a \(C^1\) variational structure while preserving
the logarithmic nonlinearity, and prove via a limiting argument that the original equation admits
infinitely many sign-changing weak solutions.


\end{abstract}

\section{Introduction}

We consider the logarithmic Schr\"odinger equation
\begin{equation}\label{eq1.1}
-\Delta u + V(x)u = u\log u^2, \quad x\in\R^N,
\end{equation}
where \(N\geq3\) and the potential \(V\in C(\R^N,\R)\) satisfies
\[
\tag{V}
0<V_0=\inf_{x\in\R^N}V(x)\leq\lim_{|x|\to\infty}V(x)=+\infty .
\]

Equation \eqref{eq1.1} is the stationary counterpart of the logarithmic Schr\"odinger dynamics proposed in nonlinear wave mechanics as a nonlinear modification of quantum evolution \cite{BialynickiBirulaMycielski1976}.
The logarithmic term is compatible with the separability requirement for noninteracting subsystems and is related to entropy-type quantities \cite{BialynickiBirulaMycielski1976,Gross1975}.
A hallmark of the logarithmic nonlinearity is the appearance of Gaussian coherent structures (gaussons) and the persistence of Gaussian profiles under the flow \cite{BialynickiBirulaMycielski1976,BialynickiBirulaMycielski1979}.
Logarithmic response laws also arise as effective models in nonlinear optics, leading to logarithmic NLS-type descriptions and Gaussian-like self-trapped states \cite{KrolikowskiEdmundsonBang2000}.
From the analytical viewpoint, $u\log u^2$ has borderline growth and is closely tied to logarithmic Sobolev inequalities and Orlicz-type variational frameworks \cite{Cazenave2003,Gross1975}.

A formal variational structure for \eqref{eq1.1} is given by the following functional
\[
I(u)
=\frac{1}{2}\int_{\R^N}\bigl[|\nabla u|^2 + (V(x)+1)u^2\bigr]\,dx
-\frac{1}{2}\int_{\R^N}u^2\log u^2\,dx,
\]
naturally associated with the space
\[
H^1_V(\R^N)
=\Bigl\{u\in H^1(\R^N): \int_{\R^N} V(x)u^2\,dx < \infty\Bigr\}.
\]
Two difficulties are intrinsic to the logarithmic term.
First, the map $u\mapsto \int_{\R^N}u^2\log u^2\,dx$ may be ill-defined on $H^1_V(\R^N)$, since there exists $u\in H^1_V(\R^N)$ with
\[
\int_{\R^N}u^2\log u^2\,dx=-\infty,
\]
so that $I(u)$ is not finite everywhere.
Second, even when finite, $I$ fails to be Fr\'echet differentiable on $H^1_V(\R^N)$ because of the singular behavior of $\log u^2$ near $u=0$.
These issues obstruct the direct use of deformation arguments and minimax schemes for sign-changing solutions, which typically require a $C^1$ functional and testing against the positive and negative parts.

There is a large literature on positive solutions, ground states, semiclassical states, and multi-bump solutions for logarithmic Schr\"odinger equations, developed through nonsmooth critical point theory, Luxemburg-type norms, penalization, constrained minimization, and Nehari-type arguments; see for instance
\cite{MR4048330,MR4097474,MR719365,MR3195154,MR3451965,MR3951962,MR3385171,MR3613541,MR3894545,MR4066104}.

For sign-changing solutions, results are more limited.
Under additional symmetry assumptions (for example, radial symmetry of $V$), infinitely many radial sign-changing solutions were obtained in \cite{MR3951962} by minimizing $I$ on suitable sign-changing Nehari-type sets and recovering the Euler--Lagrange equation via directional derivatives.
Beyond symmetric settings, the main obstruction is the lack of a usable $C^1$ variational structure for the original functional. In a significant contribution \cite{MR3894545}, Wang and Zhang obtained a profound connection between the power-law nonlinearity equation and the logarithmic Schr\"{o}dinger equation. Inspired by this work, Zhang and Wang \cite{MR4066104} proposed a global perturbation of the variational functional with power nonlinearity
\[
\begin{aligned}
I_r(u)
&=\frac12\int_{\mathbb R^N}\bigl[|\nabla u|^2+V(x)u^2\bigr]\,dx  \\
&\quad -\int_{\mathbb R^N}
\left(\frac{2}{r(r-2)}|u|^r-\frac{1}{r-2}u^2\right)\,dx
\quad\text{with}\quad r\in(2,2^{\ast}) .
\end{aligned}
\]
As a result, a family of $C^2$ functionals $I_r$ on $H^1(\mathbb{R}^N)$ was considered. Critical points of $I_r$ with uniformly bounded energy converge to critical points of $I$ as $r\to2^{+}$. This perturbation was then used to study concentration of sign-changing solutions for the logarithmic Schr\"{o}dinger equation.

We aim to obtain infinitely many sign-changing solutions of higher topological type for \eqref{eq1.1} under \((V)\), without imposing symmetry. The approach is based on the \(L^p\)-regularization perturbation recently developed in \cite{huang2026classlogarithmicschrodingerequations}, which restores smoothness while preserving the logarithmic structure; our contribution is to combine this perturbation with cone-invariant minimax methods to obtain infinitely many sign-changing solutions.

For \(\lambda\in(0,1]\), consider the perturbed equation
\begin{equation}\label{eq1.2}
\lambda |u|^{p-2}u-\Delta u+V(x)u=u\log u^2,
\quad x\in\R^N,
\end{equation}
where \(p\in\left(\max\{\frac{2N-4}{N+2},1\},2\right)\) is fixed.
Its natural energy space is
\[
X=L^{p}(\R^N)\cap H^1_V(\R^N),
\]
endowed with $\|u\|=|u|_{p}+\|u\|_{H^1_V}$.
The additional $L^{p}$-integrability with $p<2$ provides the missing integrability needed to control the negative part of $u^2\log u^2$, and yields a well-defined $C^1$ functional $I_\lambda$ on $X$.
The resulting $C^1$ functional permits the use of the invariant-set method for descending flows and the genus framework developed in \cite{MR2149532,MR3311905} to produce sign-changing critical points for $I_\lambda$.

For each fixed $\lambda\in(0,1]$ we obtain infinitely many sign-changing critical points $\{\omega_{\lambda,i}\}_{i\ge1}\subset X$, with $I_\lambda(\omega_{\lambda,i})\to+\infty$ as $i\to\infty$.
The construction is affected by a feature absent from the usual power-type nonlinearities: the map
\[
g(t)=t\log t^{2}
\]
is not monotone on $(0,+\infty)$ since
\[
g'(t)=2\log t+2<0 \quad \text{for } 0<t<e^{-1}.
\]
Consequently, the cone-invariance estimates usually used for sign-changing solutions may fail for the logarithmic nonlinearity. This is one of the main obstacles in the present minimax scheme.

The remaining step is the limiting procedure $\lambda\to0^+$.
Using uniform estimates independent of $\lambda$ and a variational perturbation argument in the spirit of \cite{MR3259554,MR2988727,MR2983045}, we prove that, for each fixed $i$ and along a sequence $\lambda_n\to0^+$, $\omega_{\lambda_n,i}$ converges strongly in $H^1_V(\R^N)$ to a sign-changing critical point $\omega_i$ of $I$.
The energy levels satisfy $I(\omega_i)\ge \alpha_i$ for some $\alpha_i>0$ with $\alpha_i\to+\infty$; consequently the limit solutions $\{\omega_i\}_{i\ge1}$ are distinct and sign-changing.

A function \(u\in H_V^1(\R^N)\) is called a weak solution of \eqref{eq1.1} if
\[
\int_{\R^N}u^2|\log u^2|\,dx<\infty
\]
and
\[
\int_{\R^N}\bigl[\nabla u\cdot\nabla\varphi+V(x)u\varphi\bigr]\,dx
=
\int_{\R^N}u\log u^2\,\varphi\,dx
\]
for every \(\varphi\in C_0^\infty(\R^N)\). In this case we write \(I'(u)=0\) in the weak sense.

The main result is the following.

\begin{Thm}\label{Thm1.1}
Assume that \((V)\) holds. Then \eqref{eq1.1} possesses infinitely many sign-changing weak solutions.
\end{Thm}

The paper is organized as follows.
In Section 2 we introduce the modified functional $I_\lambda$, verify its smoothness and compactness properties, and apply the one-pair case of the invariant-set minimax theorem in \cite{MR3311905} to obtain infinitely many sign-changing weak solutions of \eqref{eq1.2}.
In Section 3 we pass to the limit $\lambda\to0^+$, prove strong convergence of the perturbed solutions, and complete the proof of Theorem~\ref{Thm1.1}.

\section{Infinitely many sign-changing critical points for the perturbed problem}

\subsection{The perturbed functional and compactness}

Let
\[
H^{1}_{V}(\R^N)
=\left\{u\in H^{1}(\R^N): \int_{\R^N}V(x)u^{2}\,dx<\infty \right\},
\]
endowed with the inner product
\[
\langle u,v\rangle_{V}
=\int_{\R^N}\bigl[\nabla u\cdot \nabla v+V(x)uv\bigr]\,dx
\]
and the induced norm \(\|u\|_{H^{1}_{V}}=\langle u,u\rangle_{V}^{1/2}\).

Define
\[
\mathcal H:H_V^1(\R^N)\to (H_V^1(\R^N))^*
\]
by
\[
\langle \mathcal H u,v\rangle
=
\int_{\R^N}\bigl[\nabla u\cdot\nabla v+V(x)uv\bigr]\,dx,
\quad u,v\in H_V^1(\R^N).
\]
For \(m\in(1,2)\), define
\[
\mathcal N_m:L^m(\R^N)\to (L^m(\R^N))^*
\]
by
\[
\langle \mathcal N_m u,v\rangle
=
\int_{\R^N}|u|^{m-2}uv\,dx,
\quad u,v\in L^m(\R^N).
\]

\begin{Lem}\label{Lem2.1}
The operator \(\mathcal H\) is bounded and strongly monotone. More precisely,
\[
\langle \mathcal H u-\mathcal H v,u-v\rangle
=
\|u-v\|_{H_V^1}^{2},
\quad u,v\in H_V^1(\R^N).
\]
Moreover, for every \(m\in(1,2)\), the operator \(\mathcal N_m\) is bounded, hemicontinuous, coercive and monotone. In particular, it is of type \((S)\): if
\[
u_n\rightharpoonup u
\quad\text{in }L^m(\R^N)
\]
and
\[
\langle \mathcal N_m u_n,u_n-u\rangle\to0,
\]
then
\[
u_n\to u
\quad\text{in }L^m(\R^N).
\]
\end{Lem}

\begin{proof}
The assertion for \(\mathcal H\) follows directly from the definition of the \(H_V^1\)-inner product. The stated properties of \(\mathcal N_m\) are standard for the duality map on \(L^m\); see, for instance, \cite[p.~181]{MR1615056}.
\end{proof}

Throughout the paper we use the continuous extensions
\[
t^2\log t^2=0,
\qquad
 t\log t^2=0
\quad\text{at }t=0.
\]
The formal energy associated with \eqref{eq1.1} is
\begin{equation*}
I(u)
=\frac{1}{2}\int_{\R^N}\bigl[|\nabla u|^{2}+(V(x)+1)u^{2}\bigr]\,dx
-\frac{1}{2}\int_{\R^N}u^{2}\log u^{2}\,dx .
\end{equation*}
By the logarithmic Sobolev inequality \cite{MR3385171}, for every \(a>0\),
\[
\int_{\R^N}u^{2}\log u^{2}\,dx
\le \frac{a^{2}}{\pi}|\nabla u|_{2}^{2}
+\bigl(\log|u|_{2}^{2}-N(1+\log a)\bigr)|u|_{2}^{2}.
\]
Thus \(I(u)>-\infty\) for every \(u\in H_V^1(\R^N)\). However, the logarithmic term may be equal to \(-\infty\) for some \(u\in H_V^1(\R^N)\); hence \(I\) is, in general, only an extended real-valued functional on \(H_V^1(\R^N)\), and it is not of class \(C^1\) there.

To recover a smooth variational structure, fix
\[
p\in\left(\max\left\{\frac{2N-4}{N+2},1\right\},2\right)
\]
and, for \(\lambda\in(0,1]\), consider the \(L^p\)-perturbed equation
\begin{equation}\label{eq2.1}
\lambda |u|^{p-2}u-\Delta u+V(x)u=u\log u^2,
\quad x\in\R^N.
\end{equation}
We work in
\[
X=L^p(\R^N)\cap H_V^1(\R^N),
\qquad
\|u\|=|u|_p+\|u\|_{H_V^1}.
\]
Then \(X\) is a separable reflexive Banach space and \(C_0^\infty(\R^N)\) is dense in \(X\).

The perturbed functional \(I_\lambda:X\to\R\) is defined by
\begin{equation}\label{eq2.2}
I_{\lambda}(u)
=\frac{\lambda}{p}\int_{\R^N}|u|^p\,dx
+\frac{1}{2}\int_{\R^N}\bigl[|\nabla u|^{2}+(V(x)+1)u^{2}\bigr]\,dx
-\frac{1}{2}\int_{\R^N}u^{2}\log u^{2}\,dx .
\end{equation}

\begin{Lem}\label{Lem2.2}
Let \(N\ge3\), \(p\in(\max\{(2N-4)/(N+2),1\},2)\), and \(\lambda\in(0,1]\). Define
\[
\Phi(u)=\frac{1}{2}\int_{\R^N}u^{2}\log u^{2}\,dx,
\quad u\in X.
\]
Then \(\Phi\in C^1(X,\R)\), and
\begin{equation}\label{eq2.3}
\langle \Phi'(u),v\rangle
=
\int_{\R^N}\bigl[u\log u^{2}+u\bigr]v\,dx,
\quad u,v\in X.
\end{equation}
Consequently, \(I_{\lambda}\in C^{1}(X,\R)\), and
\begin{equation}\label{eq2.4}
\langle I_{\lambda}'(u),v\rangle
=\lambda\int_{\R^N}|u|^{p-2}uv\,dx
+\int_{\R^N}\bigl[\nabla u\cdot \nabla v+V(x)u v\bigr]\,dx
-\int_{\R^N}u\log u^{2}\,v\,dx .
\end{equation}
\end{Lem}

\begin{proof}
Since \(X\hookrightarrow L^p(\R^N)\cap L^{2^*}(\R^N)\), interpolation gives
\[
X\hookrightarrow L^q(\R^N)
\quad\text{for every }q\in[p,2^*].
\]
We use the elementary estimate
\begin{equation}\label{eq2.5}
t^{2}|\log t^{2}|
\le C\bigl(|t|^{p}+|t|^{2^{*}}\bigr),
\qquad
|t\log t^{2}|+|t|
\le C\bigl(|t|^{p-1}+|t|^{2^{*}-1}\bigr),
\end{equation}
valid for all \(t\in\R\). It follows that \(\Phi\) is finite in \(X\).

Let \(g(t)=\frac12t^2\log t^2\), with \(g(0)=0\). Then \(g\in C^1(\R)\) and
\(g'(t)=t\log t^2+t\). For \(u,v\in X\), the mean value theorem and \eqref{eq2.5} give, for \(|s|\le1\),
\[
\left|
\frac{g(u+sv)-g(u)}{s}
\right|
\le
C\bigl(|u|^{p-1}+|v|^{p-1}+|u|^{2^*-1}+|v|^{2^*-1}\bigr)|v|.
\]
The right-hand side belongs to \(L^1(\R^N)\). Hence dominated convergence yields
\[
\lim_{s\to0}\frac{\Phi(u+sv)-\Phi(u)}{s}
=
\int_{\R^N}(u\log u^2+u)v\,dx .
\]

It remains to prove the continuity of \(\Phi'\). Let \(u_n\to u\) in \(X\), and set
\(f(t)=t\log t^2+t\). Choose \(\chi\in C^1(\R,[0,1])\) such that \(\chi(t)=1\) for \(|t|\le1\) and \(\chi(t)=0\) for \(|t|\ge2\). Put
\[
f_1(t)=\chi(t)f(t),
\qquad
f_2(t)=(1-\chi(t))f(t).
\]
Then \(f_1,f_2\) are continuous and, by \eqref{eq2.5},
\[
|f_1(t)|\le C|t|^{p-1},
\qquad
|f_2(t)|\le C|t|^{2^*-1}.
\]
Since \(u_n\to u\) in \(L^p(\R^N)\cap L^{2^*}(\R^N)\), Vitali's theorem gives
\[
f_1(u_n)\to f_1(u)
\quad\text{in }L^{p/(p-1)}(\R^N),
\]
and
\[
f_2(u_n)\to f_2(u)
\quad\text{in }L^{2^*/(2^*-1)}(\R^N).
\]
Therefore, for \(v\in X\),
\[
\left|\int_{\R^N}(f(u_n)-f(u))v\,dx\right|
\le
|f_1(u_n)-f_1(u)|_{p/(p-1)}|v|_p
+
|f_2(u_n)-f_2(u)|_{2^*/(2^*-1)}|v|_{2^*},
\]
which tends to \(0\) uniformly for \(\|v\|\le1\). Thus \(\Phi'\) is continuous from \(X\) to \(X^*\). The formula for \(I_\lambda'\) follows from \eqref{eq2.2} and \eqref{eq2.3}.
\end{proof}

Consequently, critical points of \(I_\lambda\) are precisely the variational weak solutions of
\eqref{eq2.1}, namely functions \(u\in X\) satisfying
\[
\lambda\int_{\R^N}|u|^{p-2}u\varphi\,dx
+\int_{\R^N}\bigl[\nabla u\cdot\nabla\varphi+V(x)u\varphi\bigr]\,dx
=
\int_{\R^N}u\log u^2\,\varphi\,dx
\]
for every \(\varphi\in X\).

We shall also use the following generalized dominated convergence theorem; see \cite[p.~77]{MR1013117}.

\begin{Lem}\label{Lem2.3}
Let \(\Omega\subset\R^N\) be measurable. Let \(f_n,g_n:\Omega\to\R\) be measurable functions such that \(f_n\to f\) a.e. in \(\Omega\), \(g_n\to g\) a.e. in \(\Omega\), \(|f_n|\le g_n\) a.e. in \(\Omega\). Assume that
\[
\int_{\Omega}g_n\,dx\to\int_{\Omega}g\,dx<\infty .
\]
Then \(f\in L^1(\Omega)\) and
\[
\int_{\Omega}f_n\,dx\to\int_{\Omega}f\,dx .
\]
\end{Lem}

We now verify the compactness property of \(I_\lambda\).

\begin{Def}\label{Def2.4}
Let \(J\in C^1(X,\R)\). A sequence \(\{u_n\}\subset X\) is called a Palais--Smale sequence for \(J\) if \(\{J(u_n)\}\) is bounded and \(J'(u_n)\to0\) in \(X^*\). We say that \(J\) satisfies the Palais--Smale condition if every Palais--Smale sequence admits a convergent subsequence in \(X\).
\end{Def}

\begin{Lem}\label{Lem2.5}
Assume that \((V)\) holds. For every fixed \(\lambda\in(0,1]\), the functional \(I_\lambda\) satisfies the Palais--Smale condition in \(X\).
\end{Lem}

\begin{proof}
Let \(\{u_n\}\subset X\) be a Palais--Smale sequence for \(I_\lambda\). Then, for some \(C>0\),
\[
|I_\lambda(u_n)|\le C,
\qquad
\|I_\lambda'(u_n)\|_{X^*}=o_n(1).
\]
From \eqref{eq2.4} and \eqref{eq2.2},
\begin{equation*}
2I_{\lambda}(u_n)-\langle I_{\lambda}'(u_n),u_n\rangle
=
\frac{(2-p)\lambda}{p}\int_{\R^N}|u_n|^p\,dx
+
|u_n|_2^2.
\end{equation*}
Hence
\begin{equation}\label{eq2.6}
\lambda |u_n|_p^p+|u_n|_2^2
\le C+o_n(1)\|u_n\|.
\end{equation}
By the logarithmic Sobolev inequality, choosing \(a>0\) so that
\(a^2/\pi\le1/2\), we have
\begin{equation}\label{eq2.7}
\int_{\R^N}u_n^2\log u_n^2\,dx
\le
\frac12|\nabla u_n|_2^2
+
C\bigl(1+|u_n|_2^2\log^+ |u_n|_2^2\bigr),
\end{equation}
where \(\log^+ r=\max\{\log r,0\}\).

Using \eqref{eq2.6}, we have
\[
|u_n|_2^2\le C+o_n(1)\|u_n\|
\le C(1+\|u_n\|)
\]
for \(n\) large. Hence, by
\[
\log^+ r\le C_\delta(1+r^\delta),\qquad r\ge0,
\]
with \(\delta\in(0,p-1)\), applied to \(r=|u_n|_2^2\), we obtain
\[
|u_n|_2^2\log^+ |u_n|_2^2
\le
C_\delta(1+\|u_n\|^{1+\delta}).
\]
Combining this with \eqref{eq2.7}, we get
\begin{equation}\label{eq2.8}
\int_{\R^N}u_n^2\log u_n^2\,dx
\le
\frac12|\nabla u_n|_2^2
+
C_\delta(1+\|u_n\|^{1+\delta}).
\end{equation}

Using \eqref{eq2.2} and \eqref{eq2.8}, for \(n\) large we get
\[
C+1
\ge
\frac{\lambda}{p}|u_n|_p^p
+
\frac14\|u_n\|_{H_V^1}^2
-
C_\delta(1+\|u_n\|^{1+\delta}).
\]
Since \(1+\delta<p<2\), the right-hand side is coercive in
\[
\|u_n\|=|u_n|_p+\|u_n\|_{H_V^1}.
\]
Therefore \(\{u_n\}\) is bounded in \(X\).

Up to a subsequence,
\[
u_n\rightharpoonup u
\quad\text{in }X.
\]
By the compact embedding induced by \((V)\), see \cite{MR1349229},
\[
u_n\to u
\quad\text{in }L^q(\R^N),
\qquad 2\le q<2^*.
\]
Since \(\{u_n\}\) is also bounded in \(L^p(\R^N)\), interpolation gives
\[
u_n\to u
\quad\text{in }L^{(p+2)/2}(\R^N).
\]
Because \(I_\lambda'(u_n)\to0\) in \(X^*\),
\begin{equation}\label{eq2.9}
\begin{aligned}
&\lambda\int_{\R^N}|u_n|^{p-2}u_n(u_n-u)\,dx  \\
&\quad+
\int_{\R^N}\bigl[\nabla u_n\cdot\nabla(u_n-u)
+V(x)u_n(u_n-u)\bigr]\,dx \\
&=\int_{\R^N}u_n\log u_n^2\,(u_n-u)\,dx+o_n(1).
\end{aligned}
\end{equation}
The restriction on \(p\) allows us to choose \(\varepsilon>0\) such that
\[
(1+\varepsilon)\frac{p+2}{p}<2^*,
\qquad
(1-\varepsilon)\frac{p+2}{p}>p.
\]
Using
\[
|t\log t^2|\le C_\varepsilon\bigl(|t|^{1-\varepsilon}+|t|^{1+\varepsilon}\bigr)
\]
and Hölder's inequality, we obtain
\[
\left|
\int_{\R^N}u_n\log u_n^2\,(u_n-u)\,dx
\right|
\le
C_\varepsilon
\left||u_n|^{1-\varepsilon}+|u_n|^{1+\varepsilon}\right|_{\frac{p+2}{p}}
|u_n-u|_{\frac{p+2}{2}}
=o_n(1).
\]
Thus \eqref{eq2.9} yields
\[
\lambda\langle\mathcal N_pu_n,u_n-u\rangle+
\langle\mathcal H u_n,u_n-u\rangle=o_n(1).
\]
Since
\[
\langle\mathcal H u_n,u_n-u\rangle=\|u_n-u\|_{H_V^1}^2+o_n(1),
\]
and, by monotonicity of \(\mathcal N_p\),
\[
\langle\mathcal N_pu_n,u_n-u\rangle\ge \langle\mathcal N_pu,u_n-u\rangle=o_n(1),
\]
we infer
\[
\langle\mathcal N_pu_n,u_n-u\rangle\to0,
\qquad
\|u_n-u\|_{H_V^1}\to0.
\]
The type \((S)\) property of \(\mathcal N_p\) then gives \(u_n\to u\) in \(L^p(\R^N)\). Hence \(u_n\to u\) in \(X\), and the Palais--Smale condition follows.
\end{proof}

\subsection{Invariant sets and the descending map}
In this subsection, we will introduce the invariant set minimax theorem in \cite[Theorem~2.5]{MR3311905}, and use it to prove some auxiliary results.
Since only one pair of invariant sets is needed
in the present scalar setting, we record the corresponding one-pair formulation.
For the convenience of the reader, we first fix the terminology used below.

An isometric involution on a Banach space \(Y\) means a map \(G:Y\to Y\) such that
\[
G^2={\rm id}_Y,
\qquad
\|G u-G v\|_Y=\|u-v\|_Y
\quad\text{for all }u,v\in Y.
\]
A functional \(f:Y\to\R\) is called \(G\)-invariant if
\[
f(G u)=f(u)\quad\text{for all }u\in Y.
\]
In our application \(G=-{\rm id}_X\); thus symmetric sets are simply sets
satisfying \(A=-A\), and \(\gamma(A)\) denotes the usual Krasnosel'skii genus,
with the convention \(\gamma(\emptyset)=0\).

\begin{Def}\label{Def2.6}
Let \(Y\) be a Banach space, let \(G:Y\to Y\) be an isometric involution, and let
\(f\in C^1(Y,\R)\) be \(G\)-invariant. For an open set \(P\subset Y\), put
\[
Q=G(P),
\quad
M=P\cap Q,
\quad
\Delta=\partial P\cap\partial Q,
\quad
W=P\cup Q.
\]
We say that \(\{P\}\) is a \(G\)-admissible family of invariant sets for \(f\)
at level \(c\) if there exist a symmetric closed neighborhood \(N\) of
\(K_c\setminus W\) with \(\gamma(N)<\infty\) and \(\varepsilon_0>0\) such that,
for every \(0<\varepsilon<\varepsilon_0\), there is a continuous map
\(\eta:Y\to Y\) satisfying
\[
\eta(\overline P)\subset\overline P,
\qquad
\eta(\overline Q)\subset\overline Q,
\qquad
\eta\circ G=G\circ\eta,
\]
\[
\eta|_{f^{c-2\varepsilon}}={\rm id},
\qquad
\eta\bigl(f^{c+\varepsilon}\setminus(N\cup W)\bigr)\subset f^{c-\varepsilon},
\]
where \(f^a=\{u\in Y:f(u)\le a\}\) and
\[
K_c=\{u\in Y:f(u)=c,\ f'(u)=0\}.
\]
\end{Def}

\begin{Thm}[{\cite[Theorem~2.5]{MR3311905}}]\label{Thm2.7}
Let \(Y\) be a Banach space with an isometric involution \(G\), and let \(f\in C^1(Y,\R)\) be a \(G\)-invariant functional. Let \(P\subset Y\) be open, and set
\[
Q=G(P),\quad M=P\cap Q,
\quad \Delta=\partial P\cap\partial Q,
\quad W=P\cup Q.
\]
Assume that \(\{P\}\) is \(G\)-admissible for \(f\) at every level \(c\ge c^*:=\inf_{u\in\Delta}f(u)\). Suppose that, for each \(n\ge2\), there exists a continuous map \(\varphi^{(n)}:B^n\to Y\), where \(B^n\) is the closed unit ball of \(\R^n\), such that
\[
\varphi^{(n)}(0)\in M,
\qquad
\varphi^{(n)}(\partial B^n)\cap M=\emptyset,
\]
\[
\max\left\{\sup_{u\in F_G}f(u),\sup_{u\in\varphi^{(n)}(\partial B^n)}f(u)\right\}<c^*,
\qquad
\varphi^{(n)}(-t)=G\varphi^{(n)}(t),
\]
where \(F_G=\{u\in Y:Gu=u\}\). Define, for \(j\ge2\),
\[
c_j=\inf_{B\in\mathcal F_j}\sup_{u\in B\setminus W}f(u),
\]
where
\[
\mathcal F_j=
\left\{
B=\varphi(B^n\setminus\mathcal Y):
\varphi\in\mathcal G_n,
\ n\ge j,~
\mathcal Y=-\mathcal Y\subset B^n\ \text{open},\
\gamma(\mathcal Y)\le n-j
\right\}
\]
and
\[
\mathcal G_n=
\left\{
\varphi\in C(B^n,Y):
\varphi(-t)=G\varphi(t),
\ \varphi(0)\in M,
\ \varphi|_{\partial B^n}=\varphi^{(n)}
\right\}.
\]
Then \(c_j\), \(j\ge2\), are critical values of \(f\) with
\[
K_{c_j}\setminus W\ne\emptyset,
\qquad
c_j\to+\infty\quad\text{as }j\to\infty.
\]
\end{Thm}

We next verify the hypotheses of Theorem~\ref{Thm2.7} for the functional
\(I_\lambda\), with
\[
Y=X,\qquad f=I_\lambda,\qquad G=-{\rm id}_X.
\]
The open set \(P\) will be chosen as a small neighborhood of the positive cone,
and \(Q=G(P)\) will be the corresponding neighborhood of the negative cone.
Thus a critical point lying outside \(W=P\cup Q\) is necessarily sign-changing.
To obtain the required deformation preserving these cone neighborhoods, we first
introduce an auxiliary descending map \(A\). Its fixed points are exactly the
critical points of \(I_\lambda\). We then prove that \(A\) is continuous, that it
gives a quantitative descent away from the critical set, and that it leaves the
chosen cone neighborhoods invariant. These facts will yield the
\(G\)-admissibility required in Theorem~\ref{Thm2.7}.

Define
\[
J_\lambda(u)=
\frac{\lambda}{p}\int_{\R^N}|u|^p\,dx
+\frac12\int_{\R^N}\bigl[|\nabla u|^2+(V(x)+1)u^2\bigr]\,dx.
\]
Then \(I_\lambda=J_\lambda-\Phi\). Choose \(C_1>0\) such that
\begin{equation}\label{eq2.10}
C_1t^{p-1}+t\log t^2\ge0,
\qquad 0\le t\le1.
\end{equation}
For \(u\in X\), define \(v=A(u)\) as the unique solution of
\begin{equation}\label{eq2.11}
\langle J_\lambda'(v),\varphi\rangle
+C_1\int_{\R^N}|v|^{p-2}v\varphi\,dx
=
\int_{\R^N}(u\log u^2+u)\varphi\,dx
+C_1\int_{\R^N}|u|^{p-2}u\varphi\,dx,
\quad \varphi\in X.
\end{equation}
The map \(A\) is odd, and its fixed points are precisely the critical points of \(I_\lambda\).

\begin{Lem}\label{Lem2.8}
The map \(A:X\to X\) defined by \eqref{eq2.11} is well defined and continuous.
\end{Lem}

\begin{proof}
For fixed \(u\in X\), consider
\[
\mathscr P_u(v)
=
J_\lambda(v)+\frac{C_1}{p}\int_{\R^N}|v|^p\,dx
-
\int_{\R^N}(u\log u^2+u)v\,dx
-
C_1\int_{\R^N}|u|^{p-2}uv\,dx.
\]
By Lemma~\ref{Lem2.2}, the two linear terms are continuous on \(X\). Moreover,
\(\mathscr P_u\) is strictly convex, coercive and weakly lower semicontinuous on
\(X\). Hence it has a unique minimizer. Since \(\mathscr P_u\in C^1(X,\R)\),
the Euler--Lagrange equation of this minimizer is exactly \eqref{eq2.11}.
Thus \(A(u)\) is well defined.

It remains to prove the continuity of \(A\). Let \(u_n\to u\) in \(X\), and set
\[
v_n=A(u_n),\qquad v=A(u).
\]
For \(\varphi\in X\), define
\[
\mathcal R_n(\varphi)
=
\int_{\R^N}(u_n\log u_n^2+u_n)\varphi\,dx
+
C_1\int_{\R^N}|u_n|^{p-2}u_n\varphi\,dx
\]
and define \(\mathcal R\) analogously with \(u\) in place of \(u_n\). By
Lemma~\ref{Lem2.2},
\[
u_n\log u_n^2+u_n\to u\log u^2+u
\quad\text{in }X^*.
\]
Also, since \(u_n\to u\) in \(L^p(\R^N)\) and \(1<p<2\),
\[
\bigl\||u_n|^{p-2}u_n-|u|^{p-2}u\bigr\|_{p/(p-1)}
\le C|u_n-u|_p^{p-1}\to0.
\]
Therefore
\[
\mathcal R_n\to\mathcal R
\quad\text{in }X^*.
\]
In particular, \(\{\mathcal R_n\}\) is bounded in \(X^*\).

Testing the equation defining \(v_n=A(u_n)\) with \(v_n\), we obtain
\[
(\lambda+C_1)|v_n|_p^p
+
\int_{\R^N}\bigl[|\nabla v_n|^2+(V(x)+1)v_n^2\bigr]\,dx
=
\mathcal R_n(v_n).
\]
Hence
\[
(\lambda+C_1)|v_n|_p^p
+\|v_n\|_{H_V^1}^2
+|v_n|_2^2
\le
C\|v_n\|_X .
\]
By Young's inequality, this implies that
\[
\{|v_n|_p\}\quad\text{and}\quad \{\|v_n\|_{H_V^1}\}
\]
are bounded. Thus \(\{v_n\}\) is bounded in \(X\).

Subtracting the equations for \(v_n\) and \(v\), and testing by \(v_n-v\), gives
\[
\begin{aligned}
&\langle J_\lambda'(v_n)-J_\lambda'(v),v_n-v\rangle
+C_1\int_{\R^N}
\bigl[|v_n|^{p-2}v_n-|v|^{p-2}v\bigr](v_n-v)\,dx
=
(\mathcal R_n-\mathcal R)(v_n-v).
\end{aligned}
\]
Since \(\mathcal R_n\to\mathcal R\) in \(X^*\) and \(\{v_n-v\}\) is bounded in
\(X\), the right-hand side is \(o_n(1)\). Therefore
\begin{equation}\label{eq2.12}
\langle J_\lambda'(v_n)-J_\lambda'(v),v_n-v\rangle
+C_1\int_{\R^N}\bigl[|v_n|^{p-2}v_n-|v|^{p-2}v\bigr](v_n-v)\,dx=o_n(1).
\end{equation}

Expanding the first term in \eqref{eq2.12}, we get
\[
\begin{aligned}
&\langle J_\lambda'(v_n)-J_\lambda'(v),v_n-v\rangle \\
&=
\lambda\int_{\R^N}
\bigl[|v_n|^{p-2}v_n-|v|^{p-2}v\bigr](v_n-v)\,dx \\
&\quad+
\int_{\R^N}
\bigl[|\nabla(v_n-v)|^2+(V(x)+1)(v_n-v)^2\bigr]\,dx .
\end{aligned}
\]
Thus \eqref{eq2.12} becomes
\[
\begin{aligned}
&\int_{\R^N}
\bigl[|\nabla(v_n-v)|^2+(V(x)+1)(v_n-v)^2\bigr]\,dx \\
&\quad
+(\lambda+C_1)
\int_{\R^N}
\bigl[|v_n|^{p-2}v_n-|v|^{p-2}v\bigr](v_n-v)\,dx
=o_n(1).
\end{aligned}
\]
Both terms on the left-hand side are nonnegative. Hence the quadratic term must
converge to zero, namely
\[
\int_{\R^N}
\bigl[|\nabla(v_n-v)|^2+(V(x)+1)(v_n-v)^2\bigr]\,dx
\to0.
\]
In particular,
\[
v_n\to v
\quad\text{in }H_V^1(\R^N).
\]
Moreover,
\[
\int_{\R^N}
\bigl[|v_n|^{p-2}v_n-|v|^{p-2}v\bigr](v_n-v)\,dx
\to0.
\]

It remains to prove the convergence in \(L^p(\R^N)\). Since \(\{v_n\}\) is bounded
in \(L^p(\R^N)\), and since \(v_n\to v\) in \(H_V^1(\R^N)\), every weak
\(L^p\)-limit of a subsequence of \(\{v_n\}\) must coincide with \(v\). Hence
\[
v_n\rightharpoonup v
\quad\text{in }L^p(\R^N).
\]
Furthermore,
\[
\begin{aligned}
\langle \mathcal N_p v_n,v_n-v\rangle
&=
\int_{\R^N}
\bigl[|v_n|^{p-2}v_n-|v|^{p-2}v\bigr](v_n-v)\,dx \\
&\quad+
\int_{\R^N}|v|^{p-2}v(v_n-v)\,dx
\to0.
\end{aligned}
\]
By the type \((S)\) property of \(\mathcal N_p\), we obtain
\[
v_n\to v
\quad\text{in }L^p(\R^N).
\]
Together with the convergence in \(H_V^1(\R^N)\), this gives
\[
v_n\to v
\quad\text{in }X.
\]
Therefore \(A\) is continuous.
\end{proof}

\begin{Lem}\label{Lem2.9}
Let \(a<b\) and \(\alpha>0\). If \(u\in X\), \(I_\lambda(u)\in[a,b]\), and
\[
\|I_\lambda'(u)\|_{X^*}\ge\alpha,
\]
then there exists \(\beta>0\) such that
\[
\|u-A(u)\|\ge\beta .
\]
\end{Lem}

\begin{proof}
Set \(v=A(u)\). From \eqref{eq2.11},
\begin{equation}\label{eq2.13}
\langle I_\lambda'(u),\varphi\rangle
=
\langle J_\lambda'(u)-J_\lambda'(v),\varphi\rangle
+C_1\int_{\R^N}\bigl[|u|^{p-2}u-|v|^{p-2}v\bigr]\varphi\,dx .
\end{equation}
For \(1<p<2\),
\[
\bigl||\xi|^{p-2}\xi-|\eta|^{p-2}\eta\bigr|
\le C|\xi-\eta|^{p-1}.
\]
Using this inequality in \eqref{eq2.13}, together with Hölder's inequality, gives
\[
\|I_\lambda'(u)\|_{X^*}
\le C\left(
\|u-v\|_{H_V^1}+|u-v|_p^{p-1}
\right)
\le C\left(
\|u-v\|+\|u-v\|^{p-1}
\right).
\]
Hence \(\|u-A(u)\|\to0\) would imply \(\|I_\lambda'(u)\|_{X^*}\to0\). The conclusion follows by contradiction.
\end{proof}

We now define the invariant neighborhoods of the positive and negative cones. Write
\[
u=u_+-u_-,
\qquad
u_+=\max\{u,0\},
\quad
u_-=\max\{-u,0\}.
\]
Let \(S>0\) be the Sobolev constant satisfying
\[
S|w|_{2^*}^2\le\int_{\R^N}|\nabla w|^2\,dx,
\qquad w\in H^1(\R^N).
\]
For \(\kappa>0\), set
\[
P_\kappa^+=\{u\in X:
 C_1|u_-|_p^p+\tfrac{p}{2(p-1)}|u_-|_2^2+S|u_-|_{2^*}^2<\kappa\},
\]
and
\[
P_\kappa^-=\{u\in X:
 C_1|u_+|_p^p+\tfrac{p}{2(p-1)}|u_+|_2^2+S|u_+|_{2^*}^2<\kappa\}.
\]
Then \(P_\kappa^\pm\) are open convex subsets of \(X\), and \(P_\kappa^-=-P_\kappa^+\). Moreover, \(P_\kappa^+\) contains all nonnegative functions, while \(P_\kappa^-\) contains all nonpositive functions.

\begin{Lem}\label{Lem2.10}
Let \(c>0\) and let \(N\) be a symmetric closed neighborhood of \(K_c\setminus W\), where
\[
K_c=\{u\in X:I_\lambda(u)=c,
\ I_\lambda'(u)=0\},
\qquad
W=P_\kappa^+\cup P_\kappa^- .
\]
Then there exist \(\varepsilon_0>0\) and \(\sigma>0\) such that
\[
\langle I_\lambda'(u),u-A(u)\rangle\ge\sigma
\]
for every
\[
u\in I_\lambda^{-1}([c-\varepsilon_0,c+\varepsilon_0])\setminus(N\cup W).
\]
\end{Lem}

\begin{proof}
Set \(v=A(u)\). Define \(T_\lambda:X\to X^*\) by
\[
T_\lambda(w)=J_\lambda'(w)+C_1\mathcal N_p w,
\qquad w\in X.
\]
By the definition of \(A\), we have
\[
I_\lambda'(u)=T_\lambda u-T_\lambda v,
\]
and hence
\begin{equation*}
\langle I_\lambda'(u),u-v\rangle
=
\langle T_\lambda u-T_\lambda v,u-v\rangle .
\end{equation*}
The right-hand side is nonnegative by monotonicity.

Suppose that the conclusion is false. Then there exists a sequence \(\{u_n\}\subset X\setminus(N\cup W)\) such that
\[
I_\lambda(u_n)\to c
\]
and, setting \(v_n=A(u_n)\),
\[
D_n:=\langle I_\lambda'(u_n),u_n-v_n\rangle
=\langle T_\lambda u_n-T_\lambda v_n,u_n-v_n\rangle\to0 .
\]
Expanding \(D_n\), we get
\[
\begin{aligned}
D_n
&=\int_{\R^N}\bigl[|\nabla(u_n-v_n)|^2
+(V(x)+1)(u_n-v_n)^2\bigr]\,dx \\
&\quad+(\lambda+C_1)
\int_{\R^N}
\bigl[|u_n|^{p-2}u_n-|v_n|^{p-2}v_n\bigr](u_n-v_n)\,dx .
\end{aligned}
\]
Both terms are nonnegative. Since \(D_n\to0\), we obtain
\[
\int_{\R^N}\bigl[|\nabla(u_n-v_n)|^2
+(V(x)+1)(u_n-v_n)^2\bigr]\,dx\to0,
\]
and hence
\[
u_n-v_n\to0
\quad\text{in }H_V^1(\R^N).
\]
Moreover,
\[
\int_{\R^N}
\bigl[|u_n|^{p-2}u_n-|v_n|^{p-2}v_n\bigr](u_n-v_n)\,dx
\to0.
\]
For \(1<p<2\), the standard scalar inequality
\[
\bigl||a|^{p-2}a-|b|^{p-2}b\bigr|^{p'}
\le C_p\bigl(|a|^{p-2}a-|b|^{p-2}b\bigr)(a-b),
\qquad p'=\frac{p}{p-1},
\]
valid for all \(a,b\in\R\), gives
\[
|u_n|^{p-2}u_n-|v_n|^{p-2}v_n
\to0
\quad\text{in }L^{p'}(\R^N).
\]
Together with the strong convergence in \(H_V^1\), this yields
\[
T_\lambda u_n-T_\lambda v_n\to0
\quad\text{in }X^*.
\]
Since \(I_\lambda'(u_n)=T_\lambda u_n-T_\lambda v_n\), we have
\[
I_\lambda'(u_n)\to0
\quad\text{in }X^*.
\]
Thus \(\{u_n\}\) is a Palais--Smale sequence for \(I_\lambda\) at level \(c\). By Lemma~\ref{Lem2.5}, passing to a subsequence,
\[
u_n\to u
\quad\text{in }X,
\]
where \(I_\lambda(u)=c\) and \(I_\lambda'(u)=0\). Hence \(u\in K_c\). Since \(u_n\notin W\) and \(W\) is open, we must have \(u\notin W\). Since \(N\) is a neighborhood of \(K_c\setminus W\), we get \(u_n\in N\) for all large \(n\), contradicting \(u_n\notin N\). This proves the lemma.
\end{proof}

\begin{Lem}\label{Lem2.11}
There exists \(\kappa_0>0\) such that, for every \(\kappa\in(0,\kappa_0)\),
\[
A(P_\kappa^\pm)\subset P_\kappa^\pm,
\qquad
A(\partial P_\kappa^\pm)\subset P_\kappa^\pm .
\]
\end{Lem}

\begin{proof}
For convenience, write
\[
D(w)=C_1|w|_p^p+\frac{p}{2(p-1)}|w|_2^2+S|w|_{2^*}^2 .
\]
It is enough to prove the assertion for \(P_\kappa^+\). Let
\(u\in\overline{P_\kappa^+}\) and set \(v=A(u)\). Testing \eqref{eq2.11} with \(-v_-\) gives
\begin{equation}\label{eq2.14}
\lambda |v_-|_p^p+
\int_{\R^N}\bigl[|\nabla v_-|^2+(V(x)+1)v_-^2\bigr]\,dx
+C_1|v_-|_p^p
=
-
\int_{\R^N}(u\log u^2+u+C_1|u|^{p-2}u)v_-\,dx .
\end{equation}
By \eqref{eq2.10}, one has
\begin{equation*}
\begin{aligned}
&-\int_{\mathbb{R}^N}\left(u\log u^2+u+C_{1}|u|^{p-2}u\right)v_{-}\,dx\\
\leq&-\int_{\{x\in\mathbb{R}^N:u_{+}(x)\leq 1\}}u_{+}\log u_{+}^{2}v_{-}\,dx+
\int_{\{x\in\mathbb{R}^N:u_{-}(x)\geq 1\}}u_{-}\log u_{-}^{2}v_{-}\,dx\\
&+\int_{\mathbb{R}^N}u_{-}v_{-}\,dx+C_{1}\int_{\mathbb{R}^N}u^{p-1}_{-}v_{-}\,dx
-C_{1}\int_{\{x\in\mathbb{R}^N:u_{+}(x)\leq 1\}}u^{p-1}_{+}v_{-}\,dx\\
\leq& C_1\int_{\mathbb{R}^N}u^{p-1}_{-}v_{-}\,dx+\int_{\mathbb{R}^N}u_{-}v_{-}\,dx
+C\int_{\mathbb{R}^N}u_{-}^{2^{\ast}-1}v_{-}\,dx .
\end{aligned}
\end{equation*}
Furthermore, using Hölder's and Young's inequalities, for any \(\varepsilon>0\),
\begin{align*}
&-
\int_{\R^N}(u\log u^2+u+C_1|u|^{p-2}u)v_-\,dx \\
&\le C_1\left(\frac{p-1}{p}|u_-|_p^p+\frac1p|v_-|_p^p\right)
+\frac12|u_-|_2^2+\frac12|v_-|_2^2
+\varepsilon |v_-|_{2^*}^2+C_\varepsilon |u_-|_{2^*}^{2(2^*-1)} .
\end{align*}
Taking \(\varepsilon=S/p\) and putting \(\alpha=(p-1)/p\), it follows from \eqref{eq2.14} that
\begin{equation*}
\begin{aligned}
&\alpha C_1|v_-|_p^p+
\frac12|v_-|_2^2+
\alpha S|v_-|_{2^*}^2 \\
&\qquad\le
\alpha C_1|u_-|_p^p+
\frac12|u_-|_2^2+
C_3|u_-|_{2^*}^2|u_-|_{2^*}^{2(2^*-2)} .
\end{aligned}
\end{equation*}
Choose \(\kappa_0>0\) so small that
\[
C_3\left(\frac{\kappa_0}{S}\right)^{2^*-2}\le \frac{\alpha S}{2} .
\]
If \(0<\kappa<\kappa_0\) and \(u\in\overline{P_\kappa^+}\), then
\(D(u_-)\le\kappa\), and therefore
\[
S|u_-|_{2^*}^2\le D(u_-)\le\kappa .
\]
Hence
\[
C_3|u_-|_{2^*}^{2(2^*-2)}
\le C_3\left(\frac{\kappa}{S}\right)^{2^*-2}
\le \frac{\alpha S}{2},
\]
and consequently
\[
D(v_-)
\le
D(u_-)-\frac{S}{2}|u_-|_{2^*}^2
\le D(u_-) .
\]
If \(u\in P_\kappa^+\), then \(D(u_-)<\kappa\), so \(D(v_-)<\kappa\). If
\(u\in\partial P_\kappa^+\), then, by the continuity of \(u\mapsto D(u_-)\),
\(D(u_-)=\kappa\). Thus \(u_-\not\equiv0\), and the preceding estimate gives
\(D(v_-)<\kappa\). In both cases \(v=A(u)\in P_\kappa^+\). Therefore
\[
A(P_\kappa^+)\subset P_\kappa^+,
\qquad
A(\partial P_\kappa^+)\subset P_\kappa^+ .
\]

For \(P_\kappa^-\), taking \(\phi=v_+\) in \eqref{eq2.11} yields the same argument with \(u_+\) and
\(v_+\) in place of \(u_-\) and \(v_-\). This proves the lemma.
\end{proof}

\begin{Lem}\label{Lem2.12}
For every fixed \(\lambda\in(0,1]\) and every \(0<\kappa<\kappa_0\), the family
\(\{P_\kappa^+\}\) is \(G\)-admissible for \(I_\lambda\) at every level \(c>0\),
where \(G=-{\rm id}_X\).
\end{Lem}

\begin{proof}
Since \(G=-{\rm id}_X\), we have
\[
G(P_\kappa^+)=P_\kappa^-.
\]
Thus, in the notation of Definition~\ref{Def2.6}, we take
\[
P=P_\kappa^+,\qquad
Q=P_\kappa^-,
\qquad
W=P_\kappa^+\cup P_\kappa^- .
\]
The set \(W\) is symmetric.

Fix \(c>0\). Since \(I_\lambda\) satisfies the Palais--Smale condition,
\(K_c\) is compact, and therefore \(K_c\setminus W\) is compact. Moreover,
\[
0\notin K_c\setminus W,
\]
because \(I_\lambda(0)=0<c\). Hence, by the standard compactness property of the
Krasnosel'skii genus, we may choose a symmetric closed neighborhood \(N\) of
\(K_c\setminus W\) such that
\[
\gamma(N)<\infty .
\]
Choose another symmetric closed neighborhood \(N_0\) of \(K_c\setminus W\) such
that
\[
N_0\subset \operatorname{int}N .
\]

By Lemma~\ref{Lem2.10}, applied with \(N_0\), there exist
\(\varepsilon_0>0\) and \(\sigma>0\) such that
\[
\langle I_\lambda'(u),u-A(u)\rangle\ge\sigma
\]
for every
\[
u\in I_\lambda^{-1}([c-\varepsilon_0,c+\varepsilon_0])
\setminus (N_0\cup W).
\]
Fix \(0<\varepsilon<\varepsilon_0/4\), and set
\[
\mathcal S
=
I_\lambda^{-1}([c-\varepsilon,c+\varepsilon])
\setminus (N_0\cup W).
\]
For \(u\in\mathcal S\), consider
\[
h_u(t)=I_\lambda((1-t)u+tA(u)),\qquad 0\le t\le1.
\]
Then
\[
h_u'(0)
=
-\langle I_\lambda'(u),u-A(u)\rangle
\le -\sigma .
\]
By the continuity of \(A\) and \(I_\lambda'\), for each \(u\in\mathcal S\) there
exist an open neighborhood \(U_u\) of \(u\) and a number \(\tau_u\in(0,1)\) such
that
\[
I_\lambda((1-t)z+tA(z))
\le
I_\lambda(z)-\frac{\sigma}{2}t,
\qquad
z\in U_u\cap\mathcal S,\quad 0\le t\le\tau_u .
\]

Using the symmetric open cover \(\{U_u,-U_u:u\in\mathcal S\}\), take a locally
finite symmetric refinement and an even continuous partition of unity
subordinate to it. The standard Euler-polygonal deformation associated with the
descent segments
\[
z\mapsto (1-t)z+tA(z),\qquad 0\le t\le1,
\]
then gives a continuous map \(\eta:X\to X\) with the following properties:
\[
\eta\bigl(I_\lambda^{c+\varepsilon}\setminus(N\cup W)\bigr)
\subset I_\lambda^{c-\varepsilon},
\qquad
\eta|_{I_\lambda^{c-2\varepsilon}}={\rm id}.
\]
Indeed, the deformation is taken to be the identity outside the active region,
and the cut-off near \(N_0\) is chosen so that, since
\(N_0\subset\operatorname{int}N\), points initially in
\(I_\lambda^{c+\varepsilon}\setminus(N\cup W)\) are not stopped before their
energy is lowered below \(c-\varepsilon\).

Since \(A\) is odd, \(I_\lambda\) is even, and the cover, refinement and cut-off
functions are chosen symmetrically, the deformation can be constructed so that
\[
\eta\circ G=G\circ\eta .
\]

It remains to check that the deformation preserves the two invariant sets in
the sense required by Definition~\ref{Def2.6}. By Lemma~\ref{Lem2.11},
\[
A(P_\kappa^\pm)\subset P_\kappa^\pm,
\qquad
A(\partial P_\kappa^\pm)\subset P_\kappa^\pm .
\]
Since
\[
\overline{P_\kappa^\pm}=P_\kappa^\pm\cup\partial P_\kappa^\pm,
\]
we obtain
\[
A(\overline{P_\kappa^\pm})
\subset P_\kappa^\pm
\subset \overline{P_\kappa^\pm}.
\]
Moreover, \(P_\kappa^\pm\) are convex, and therefore their closures are convex.
Consequently, for every \(z\in\overline{P_\kappa^\pm}\) and every \(t\in[0,1]\),
\[
(1-t)z+tA(z)\in\overline{P_\kappa^\pm}.
\]
Thus each polygonal segment used in the deformation leaves
\(\overline{P_\kappa^\pm}\) invariant, and so do all finite concatenations of
such segments. Hence
\[
\eta(\overline{P_\kappa^+})\subset\overline{P_\kappa^+},
\qquad
\eta(\overline{P_\kappa^-})\subset\overline{P_\kappa^-}.
\]

All conditions in Definition~\ref{Def2.6} are verified. Therefore
\(\{P_\kappa^+\}\) is \(G\)-admissible for \(I_\lambda\) at level \(c\).
Since \(c>0\) was arbitrary, the proof is complete.
\end{proof}

\subsection{Minimax classes and critical points of the perturbed problem}

Choose a linearly independent sequence \(\{\varphi_j\}_{j\ge1}\subset C_0^\infty(\R^N)\) which is dense in \(X\), and set
\[
D_n=\operatorname{span}\{\varphi_1,\ldots,\varphi_n\}.
\]
Applying Gram--Schmidt in \(H_V^1(\R^N)\), we obtain an \(H_V^1\)-orthonormal sequence
\(\{e_j\}\subset C_0^\infty(\R^N)\) such that
\[
D_n=\operatorname{span}\{e_1,\ldots,e_n\}.
\]
For \(n\in\N\), set
\[
X_n=X\cap D_n^{\perp_V}
=\{u\in X:\langle u,e_j\rangle_V=0,
\ j=1,\ldots,n\}.
\]
Fix \(\delta\in(0,2^*-2)\) throughout this subsection. Let
\[
D(w)=C_1|w|_p^p+\frac{p}{2(p-1)}|w|_2^2+S|w|_{2^*}^2
\]
and
\[
\mathcal M_0=
\{u\in X:D(u_+)\le\kappa_0,
\ D(u_-)\le\kappa_0\}.
\]
By interpolation between \(L^2(\R^N)\) and \(L^{2^*}(\R^N)\), there exists \(R_0>0\), depending only on \(\kappa_0\), such that
\begin{equation}\label{eq2.16}
|u|_{2+\delta}\le R_0,
\quad u\in\mathcal M_0.
\end{equation}

\begin{Lem}\label{Lem2.13}
There exist sequences \(\{s_i\}_{i\ge1}\subset(0,+\infty)\) and \(\{\tilde s_i\}_{i\ge1}\subset\R\), independent of \(\lambda\), such that
\[
s_i\to+\infty,
\qquad
\tilde s_i\to+\infty,\qquad\text{as}\quad i\to\infty
\]
and
\[
I_\lambda(u)\ge I(u)\ge\tilde s_i
\]
for every \(u\in X_i\) satisfying \(|u|_{2+\delta}=s_i\), where
\[
I(u)
=
\frac12\int_{\R^N}\bigl[|\nabla u|^2+(V(x)+1)u^2\bigr]\,dx
-
\frac12\int_{\R^N}u^2\log u^2\,dx .
\]
\end{Lem}

\begin{proof}
Set
\[
\theta_i=
\sup\left\{
|u|_{2+\delta}:u\in X_i,
\|u\|_{H_V^1}=1
\right\}.
\]
We first prove that \(\theta_i\to0\). Suppose by contradiction that this is false. Then there exist
\(\varepsilon_0>0\) and a subsequence \(i_k\to\infty\) such that \(\theta_{i_k}\ge2\varepsilon_0\). By the definition of the supremum, we can choose functions \(u_k\in X_{i_k}\) such that
\[
\|u_k\|_{H_V^1}=1,
\qquad
|u_k|_{2+\delta}\ge\varepsilon_0 .
\]
Passing to a subsequence, we may assume that \(u_k\rightharpoonup u\) in \(H_V^1(\R^N)\). For each fixed \(m\in\N\), we have \(i_k\ge m\) for all large \(k\), and hence
\[
\langle u_k,e_m\rangle_V=0
\quad\text{for all large }k.
\]
Letting \(k\to\infty\), we obtain \(\langle u,e_m\rangle_V=0\) for every \(m\in\N\). Since
\(\bigcup_mD_m\) is dense in \(H_V^1(\R^N)\), it follows that \(u=0\). The compact embedding induced by \((V)\) then gives
\[
u_k\to0
\quad\text{in }L^{2+\delta}(\R^N),
\]
which contradicts \(|u_k|_{2+\delta}\ge\varepsilon_0\). Therefore \(\theta_i\to0\).
Moreover, since \(X_{i+1}\subset X_i\), the sequence \(\{\theta_i\}\) is nonincreasing.

Let
\[
s_i=R_0+\theta_i^{-1/2}.
\]
Then \(s_i\to+\infty\), and \(\{s_i\}\) is nondecreasing. If \(u\in X_i\) and \(|u|_{2+\delta}=s_i\), then
\[
\|u\|_{H_V^1}\ge\frac{s_i}{\theta_i}.
\]
Moreover, since \(\log t^2\le(2/\delta)t^\delta\) for \(t\ge1\),
\[
\int_{\{|u|\ge1\}}u^2\log u^2\,dx
\le\frac{2}{\delta}|u|_{2+\delta}^{2+\delta}.
\]
Thus
\[
I(u)
\ge
\frac12\left(\frac{s_i}{\theta_i}\right)^2
-
\frac1\delta s_i^{2+\delta}
=:\tilde s_i.
\]
Since \(s_i=R_0+\theta_i^{-1/2}\) and \(\delta<2^*-2\le4\), we have \(\tilde s_i\to+\infty\). The inequality \(I_\lambda\ge I\) follows from the nonnegativity of the perturbation term.
\end{proof}

\begin{Lem}\label{Lem2.14}
For every \(n\ge2\), there exists an odd continuous map
\(\varphi^{(n)}:B^n\to X\), where \(B^n\) is the closed unit ball of
\(\R^n\), such that
\[
\varphi^{(n)}(0)=0,
\qquad
\sup_{t\in\partial B^n}I_\lambda(\varphi^{(n)}(t))<0
\quad\text{for all }\lambda\in(0,1],
\]
and
\[
|\varphi^{(n)}(t)|_{2+\delta}>s_n
\quad\text{for every }t\in\partial B^n.
\]
\end{Lem}

\begin{proof}
Choose \(u_1,\ldots,u_n\in C_0^\infty(\R^N)\setminus\{0\}\) with mutually
disjoint supports. For \(t=(t_1,\ldots,t_n)\in B^n\), set
\[
w_t=\sum_{j=1}^n t_j u_j .
\]
Let \(R>0\) be a parameter to be chosen later, and define
\[
\varphi^{(n)}(t)=R w_t
=
R\sum_{j=1}^n t_j u_j,
\qquad t\in B^n .
\]
Then \(\varphi^{(n)}\) is continuous and odd, and
\[
\varphi^{(n)}(0)=0.
\]

We first estimate the energy on \(\partial B^n\). Since the supports of
\(u_1,\ldots,u_n\) are mutually disjoint, for \(t\in\partial B^n\) we have
\[
|w_t|_2^2
=
\sum_{j=1}^n t_j^2 |u_j|_2^2
\ge
C_{n,1}>0.
\]
Moreover, since \(\partial B^n\) is compact and \(t\mapsto w_t\) takes values
in the finite-dimensional space \(\operatorname{span}\{u_1,\ldots,u_n\}\), there
exist constants \(C_{n,2},C_{n,3},C_{n,4}>0\) such that, for all
\(t\in\partial B^n\),
\[
|w_t|_p^p\le C_{n,2},
\]
\[
\int_{\R^N}\bigl[|\nabla w_t|^2+(V(x)+1)w_t^2\bigr]\,dx
\le C_{n,3},
\]
and
\[
\left|\int_{\R^N}w_t^2\log w_t^2\,dx\right|\le C_{n,4}.
\]
Therefore, for \(t\in\partial B^n\) and \(\lambda\in(0,1]\),
\[
\begin{aligned}
I_\lambda(Rw_t)
&=
\frac{\lambda}{p}R^p|w_t|_p^p
+\frac{R^2}{2}\int_{\R^N}\bigl[|\nabla w_t|^2+(V(x)+1)w_t^2\bigr]\,dx
-\frac{R^2}{2}\int_{\R^N}w_t^2\log(R^2w_t^2)\,dx \\
&=
\frac{\lambda}{p}R^p|w_t|_p^p
+\frac{R^2}{2}\int_{\R^N}\bigl[|\nabla w_t|^2+(V(x)+1)w_t^2\bigr]\,dx
-\frac{R^2}{2}\log R^2\,|w_t|_2^2
-\frac{R^2}{2}\int_{\R^N}w_t^2\log w_t^2\,dx \\
&\le
R^2\left(
C_{n,2}R^{p-2}
+C_{n,3}
-\frac{C_{n,1}}{2}\log R^2
\right),
\end{aligned}
\]
after increasing \(C_{n,2}\) and \(C_{n,3}\) if necessary. Since \(p<2\), we have
\(R^{p-2}\to0\) as \(R\to+\infty\), while \(\log R^2\to+\infty\). Hence
\[
\sup_{t\in\partial B^n}I_\lambda(\varphi^{(n)}(t))<0
\quad\text{for all }\lambda\in(0,1]
\]
provided \(R>0\) is chosen sufficiently large.

It remains to prove the \(L^{2+\delta}\)-lower bound. The map
\[
t\mapsto |w_t|_{2+\delta}
\]
is continuous on \(\partial B^n\). Moreover, \(w_t\not\equiv0\) for
\(t\in\partial B^n\), because the functions \(u_1,\ldots,u_n\) have mutually
disjoint supports and are nonzero. Hence
\[
d_n:=
\min_{t\in\partial B^n}|w_t|_{2+\delta}>0.
\]
Thus, for \(t\in\partial B^n\),
\[
|\varphi^{(n)}(t)|_{2+\delta}
=
R|w_t|_{2+\delta}
\ge
Rd_n .
\]
Increasing \(R\), if necessary, so that
\[
R>\frac{s_n}{d_n},
\]
we obtain
\[
|\varphi^{(n)}(t)|_{2+\delta}>s_n
\quad\text{for every }t\in\partial B^n.
\]
Taking \(R>0\) large enough to satisfy both requirements completes the proof.
\end{proof}

For \(\kappa>0\), put
\[
\Delta_\kappa=\partial P_\kappa^+\cap\partial P_\kappa^-.
\]

\begin{Lem}\label{Lem2.15}
For each fixed \(\lambda\in(0,1]\), there exists \(\kappa_1=\kappa_1(\lambda)>0\) such that, if
\[
0<\kappa<\min\{\kappa_0,\kappa_1\},
\]
then
\[
c_\lambda^*:=\inf_{u\in\Delta_\kappa}I_\lambda(u)>0.
\]
\end{Lem}

\begin{proof}
Recall that
\[
D(w)=C_1|w|_p^p+\frac{p}{2(p-1)}|w|_2^2+S|w|_{2^*}^2 .
\]
If \(u\in\Delta_\kappa\), then, by the definition of \(P_\kappa^\pm\),
\[
D(u_+)=D(u_-)=\kappa .
\]
Since \(u_+\) and \(u_-\) have disjoint supports, it is enough to estimate the energy on each nodal part.

For every \(w\in X\), using \(t^2\log t^2<0\) for \(0<|t|<1\) and
\(t^2\log t^2\le C|t|^{2^*}\) for \(|t|\ge1\), we have
\[
\int_{\R^N}w^2\log w^2\,dx
\le
\int_{\{|w|\ge1\}}w^2\log w^2\,dx
\le C|w|_{2^*}^{2^*}.
\]
Hence, by \(V\ge V_0>0\) and the Sobolev inequality,
\[
\begin{aligned}
&\frac{\lambda}{p}|w|_p^p
+\frac12\int_{\R^N}\bigl[|\nabla w|^2+(V(x)+1)w^2\bigr]\,dx
-\frac12\int_{\R^N}w^2\log w^2\,dx \\
&\qquad\ge
\frac{\lambda}{p}|w|_p^p
+\frac12|w|_2^2
+\frac12 S|w|_{2^*}^2
-C|w|_{2^*}^{2^*}.
\end{aligned}
\]
Set
\[
\mu_\lambda=\min\left\{\frac{\lambda}{pC_1},\frac{p-1}{p},\frac12\right\}>0.
\]
Then
\[
\frac{\lambda}{p}|w|_p^p
+\frac12|w|_2^2
+\frac12 S|w|_{2^*}^2
\ge \mu_\lambda D(w).
\]
Moreover, if \(D(w)\le\kappa\), then
\[
S|w|_{2^*}^2\le D(w)\le\kappa,
\]
and therefore
\[
|w|_{2^*}^{2^*}
=\bigl(|w|_{2^*}^2\bigr)^{2^*/2}
\le C\kappa^{2^*/2}.
\]
Consequently, whenever \(D(w)=\kappa\),
\[
\frac{\lambda}{p}|w|_p^p
+\frac12\int_{\R^N}\bigl[|\nabla w|^2+(V(x)+1)w^2\bigr]\,dx
-\frac12\int_{\R^N}w^2\log w^2\,dx
\ge
\mu_\lambda\kappa-C\kappa^{2^*/2}.
\]
Since \(2^*/2>1\), we can choose \(\kappa_1=\kappa_1(\lambda)>0\) so small that
\[
C\kappa^{2^*/2}\le \frac{\mu_\lambda}{2}\kappa
\quad\text{for all }0<\kappa<\kappa_1.
\]
Thus, if \(0<\kappa<\min\{\kappa_0,\kappa_1\}\) and \(D(w)=\kappa\), then
\[
\frac{\lambda}{p}|w|_p^p
+\frac12\int_{\R^N}\bigl[|\nabla w|^2+(V(x)+1)w^2\bigr]\,dx
-\frac12\int_{\R^N}w^2\log w^2\,dx
\ge
\frac{\mu_\lambda}{2}\kappa.
\]
Applying this estimate to \(w=u^+\) and \(w=u^-\), and using the disjointness of their supports, we obtain
\[
I_\lambda(u)\ge \mu_\lambda\kappa>0
\quad\text{for every }u\in\Delta_\kappa.
\]
Therefore
\[
c_\lambda^*=\inf_{u\in\Delta_\kappa}I_\lambda(u)\ge \mu_\lambda\kappa>0.
\]
The proof is complete.
\end{proof}

For the rest of this subsection, fix \(\lambda\in(0,1]\) and choose
\[
0<\kappa<\min\{\kappa_0,\kappa_1(\lambda)\}.
\]
Set
\[
G=-{\rm id}_X,
\quad
P=P_\kappa^+,
\quad
Q=P_\kappa^-,
\quad
M=P\cap Q,
\quad
\Delta=\partial P\cap\partial Q,
\quad
W=P\cup Q.
\]
For \(j\ge2\), define
\[
c_{\lambda,j}=\inf_{B\in\mathcal F_j}\sup_{u\in B\setminus W}I_\lambda(u),
\]
where
\[
\mathcal F_j=
\left\{
B=\varphi(B^n\setminus\mathcal Y):
\varphi\in\mathcal G_n,
\ n\ge j,~
\mathcal Y=-\mathcal Y\subset B^n\ \text{open},~
\gamma(\mathcal Y)\le n-j
\right\}
\]
and
\[
\mathcal G_n=
\left\{
\varphi\in C(B^n,X):
\varphi(-t)=-\varphi(t),
\ \varphi(0)\in M,
\ \varphi|_{\partial B^n}=\varphi^{(n)}
\right\}.
\]

\begin{Lem}\label{Lem2.16}
The classes \(\mathcal F_j\) are nonempty for every \(j\ge2\), and the assumptions on the maps \(\varphi^{(n)}\) in Theorem~\ref{Thm2.7} are satisfied.
\end{Lem}

\begin{proof}
The maps \(\varphi^{(n)}\) from Lemma~\ref{Lem2.14} are odd and satisfy \(\varphi^{(n)}(0)=0\in M\). Since \(G=-{\rm id}_X\), the fixed point set of \(G\) is
\[
F_G=\{u\in X:Gu=u\}=\{0\}.
\]
Hence
\[
\sup_{u\in F_G} I_\lambda(u)=I_\lambda(0)=0<c_\lambda^*.
\]
Moreover, Lemma~\ref{Lem2.14} gives
\[
\sup_{t\in\partial B^n}I_\lambda(\varphi^{(n)}(t))<0<c_\lambda^*.
\]
Therefore the energy condition in Theorem~\ref{Thm2.7} holds. Finally, by taking \(R\) larger in Lemma~\ref{Lem2.14} if necessary, we ensure \(\varphi^{(n)}(\partial B^n)\cap M=\emptyset\). Indeed, for all \(0<\kappa<\kappa_0\), the set \(M=P_\kappa^+\cap P_\kappa^-\) is contained in \(\mathcal M_0\), hence is bounded in \(L^{2+\delta}(\R^N)\) by \eqref{eq2.16}, whereas \(|\varphi^{(n)}(t)|_{2+\delta}\to\infty\) uniformly on \(\partial B^n\) as \(R\to\infty\). Thus \(\mathcal F_j\ne\emptyset\).
\end{proof}

\begin{Lem}\label{Lem2.17}
For every \(j\ge3\),
\[
c_{\lambda,j}\ge\tilde s_{j-2}.
\]
In particular this lower bound is independent of \(\lambda\) and tends to \(+\infty\) as \(j\to\infty\).
\end{Lem}

\begin{proof}
Fix \(j\ge3\), and set \(m=j-2\). Let \(B\in\mathcal F_j\). Then
\[
B=\varphi(B^n\setminus\mathcal Y)
\]
for some \(n\ge j\), \(\varphi\in\mathcal G_n\), and some symmetric open set \(\mathcal Y\subset B^n\) with \(\gamma(\mathcal Y)\le n-j\). We prove that
\begin{equation}\label{eq2.17}
(B\setminus W)
\cap X_m
\cap
\{u\in X: |u|_{2+\delta}=s_m\}
\ne\emptyset .
\end{equation}

Since \(\varphi(0)\in M\subset\mathcal M_0\), \eqref{eq2.16} gives
\[
|\varphi(0)|_{2+\delta}\le R_0<s_m.
\]
On the other hand, \(\varphi=\varphi^{(n)}\) on \(\partial B^n\), and Lemma~\ref{Lem2.14} gives
\[
|\varphi(t)|_{2+\delta}>s_n\ge s_m,
\qquad t\in\partial B^n.
\]
Thus
\[
\Omega=
\{t\in B^n: |\varphi(t)|_{2+\delta}<s_m\}
\]
is a symmetric open neighborhood of \(0\) whose closure is contained in the interior of \(B^n\). Put
\[
Z=\partial\Omega,
\]
where the boundary is taken in \(\R^n\). Then \(Z\) is symmetric and, since \(\overline\Omega\subset\operatorname{int}B^n\), the continuity of \(t\mapsto|\varphi(t)|_{2+\delta}\) gives
\[
Z\subset
\{t\in B^n: |\varphi(t)|_{2+\delta}=s_m\}.
\]
By Borsuk's theorem,
\begin{equation}\label{eq2.18}
\gamma(Z)=n.
\end{equation}

Let
\[
Z_W=Z\cap\varphi^{-1}(W).
\]
We claim that
\begin{equation}\label{eq2.19}
\gamma(Z_W)\le1.
\end{equation}
Indeed, if \(t\in Z_W\), then \(\varphi(t)\in P_\kappa^+\cup P_\kappa^-\) and
\[
|\varphi(t)|_{2+\delta}=s_m>R_0,
\]
so \(\varphi(t)\notin\mathcal M_0\). Consequently, if \(\varphi(t)\in P_\kappa^+\), then
\(\varphi(t)\notin\overline{P_\kappa^-}\); otherwise both \(D(\varphi(t)_-)<\kappa<\kappa_0\) and \(D(\varphi(t)_+)\le\kappa<\kappa_0\), hence \(\varphi(t)\in\mathcal M_0\), a contradiction. Similarly, if \(\varphi(t)\in P_\kappa^-\), then \(\varphi(t)\notin\overline{P_\kappa^+}\).

Consider
\[
\rho(u)=\operatorname{dist}(u,P_\kappa^-)-\operatorname{dist}(u,P_\kappa^+).
\]
Since \(P_\kappa^-=-P_\kappa^+\), the map \(\rho\) is odd. By the previous paragraph,
\(\rho(\varphi(t))\ne0\) for \(t\in Z_W\). Hence \(t\mapsto \rho(\varphi(t))\) is an odd continuous map from \(Z_W\) into \(\R\setminus\{0\}\), proving \eqref{eq2.19}.

Set
\[
A=Z\setminus(\mathcal Y\cup Z_W).
\]
We show that \(\gamma(A)\ge j-1\). If \(\gamma(A)\le j-2\), then by subadditivity of the genus,
\[
\gamma(Z)
\le
\gamma(Z\cap\mathcal Y)+\gamma(Z_W)+\gamma(A)
\le
(n-j)+1+(j-2)=n-1,
\]
contradicting \eqref{eq2.18}. Therefore
\begin{equation}\label{eq2.20}
\gamma(A)\ge j-1.
\end{equation}

Let \(\mathcal P_m:H_V^1(\R^N)\to D_m\) be the \(H_V^1\)-orthogonal projection. If
\[
\mathcal P_m\varphi(t)\ne0
\quad\text{for every }t\in A,
\]
then the odd continuous map
\[
t\mapsto
\bigl(
\langle\varphi(t),e_1\rangle_V,
\ldots,
\langle\varphi(t),e_m\rangle_V
\bigr)
\]
maps \(A\) into \(\R^m\setminus\{0\}\). Hence \(\gamma(A)\le m=j-2\), contradicting \eqref{eq2.20}. Thus there exists \(t_0\in A\) such that \(\mathcal P_m\varphi(t_0)=0\). Set \(u=\varphi(t_0)\). Then \(u\in B\), since \(t_0\notin\mathcal Y\). Moreover, because
\(t_0\in Z\) and \(t_0\notin Z_W=Z\cap\varphi^{-1}(W)\), we have
\(u=\varphi(t_0)\notin W\). Also, \(u\in X_m\), since
\(\mathcal P_m u=0\), and \(|u|_{2+\delta}=s_m\), since \(t_0\in Z\). This proves \eqref{eq2.17}.

By Lemma~\ref{Lem2.13}, for this \(u\) we have
\[
I_\lambda(u)\ge\tilde s_m=\tilde s_{j-2}.
\]
Therefore, for every \(B\in\mathcal F_j\),
\[
\sup_{u\in B\setminus W}I_\lambda(u)\ge\tilde s_{j-2}.
\]
Taking the infimum over \(B\in\mathcal F_j\), we obtain \(c_{\lambda,j}\ge\tilde s_{j-2}\). The conclusion follows from \(\tilde s_i\to+\infty\).
\end{proof}

We are now ready to apply the invariant-set minimax theorem to the perturbed
functional \(I_\lambda\). The preceding lemmas verify the required hypotheses:
Lemma~\ref{Lem2.5} gives the Palais--Smale condition, Lemma~\ref{Lem2.12}
gives the \(G\)-admissibility of the cone pair, Lemmas~\ref{Lem2.15} and
\ref{Lem2.16} provide the minimax geometry, and Lemma~\ref{Lem2.17} gives the
uniform lower growth of the minimax levels. Hence the abstract critical values
defined above are realized by critical points lying outside the two cone
neighborhoods, and therefore correspond to sign-changing solutions of the
perturbed problem.

\begin{Thm}\label{Thm2.18}
For every fixed \(\lambda\in(0,1]\), the functional \(I_\lambda\) admits a sequence of sign-changing critical points \(\{\omega_{\lambda,j}\}_{j\ge2}\subset X\) such that
\[
I_\lambda(\omega_{\lambda,j})=c_{\lambda,j},
\qquad
c_{\lambda,j}\to+\infty.
\]
Moreover, for \(j\ge3\),
\[
c_{\lambda,j}\ge\tilde s_{j-2}.
\]
\end{Thm}

\begin{proof}
We apply Theorem~\ref{Thm2.7} with \(Y=X\), \(f=I_\lambda\), \(G=-{\rm id}_X\), \(P=P_\kappa^+\), and \(Q=P_\kappa^-\). Its assumptions follow from Lemmas~\ref{Lem2.5}, \ref{Lem2.12}, \ref{Lem2.15}, and \ref{Lem2.16}. Therefore Theorem~\ref{Thm2.7} gives a critical point
\[
\omega_{\lambda,j}\in X\setminus(P\cup Q)
\]
at level \(c_{\lambda,j}\) for every \(j\ge3\). Since \(P_\kappa^+\) contains all nonnegative functions and \(P_\kappa^-\) contains all nonpositive functions, every critical point in \(X\setminus(P\cup Q)\) is sign-changing. Finally, Lemma~\ref{Lem2.17} gives \(c_{\lambda,j}\to+\infty\).
\end{proof}

\section{Proof of Theorem~\texorpdfstring{\ref{Thm1.1}}{1.1}}

We first record the uniform bounds for the minimax levels obtained in the previous section.

\begin{Lem}\label{Lem3.1}
For every \(j\) sufficiently large, there exist constants
\(\alpha_j,\beta_j>0\), independent of \(\lambda\in(0,1]\), such that
\[
\alpha_j\le c_{\lambda,j}=I_\lambda(\omega_{\lambda,j})\le \beta_j .
\]
More precisely, there exists \(j_0\ge3\) such that the above estimate holds for
all \(j\ge j_0\), and
\[
\alpha_j\to+\infty
\quad\text{as }j\to+\infty .
\]
\end{Lem}

\begin{proof}
The lower bound follows from Lemma~\ref{Lem2.17}. For every \(j\ge3\),
\[
c_{\lambda,j}\ge \tilde s_{j-2},
\]
where \(\tilde s_i\to+\infty\) and the sequence \(\{\tilde s_i\}_{i\ge1}\) is
independent of \(\lambda\). Hence there exists \(j_0\ge3\) such that
\[
\tilde s_{j-2}>0
\quad\text{for all }j\ge j_0.
\]
For \(j\ge j_0\), set
\[
\alpha_j=\tilde s_{j-2}.
\]
Then
\[
\alpha_j>0,\qquad
c_{\lambda,j}\ge\alpha_j,
\]
and
\[
\alpha_j\to+\infty
\quad\text{as }j\to+\infty .
\]

It remains to prove the upper bound. Define
\[
\widetilde I(u)
=
\frac1p\int_{\R^N}|u|^p\,dx
+\frac12\int_{\R^N}\bigl[|\nabla u|^2+(V(x)+1)u^2\bigr]\,dx
-\frac12\int_{\R^N}u^2\log u^2\,dx .
\]
Since \(0<\lambda\le1\), we have
\[
I_\lambda(u)\le \widetilde I(u),
\qquad u\in X.
\]

For each fixed \(j\ge j_0\), take the map
\(\varphi^{(j)}:B^j\to X\) constructed in Lemma~\ref{Lem2.14}. This map is
independent of \(\lambda\), because the estimates in Lemma~\ref{Lem2.14} are
uniform for \(\lambda\in(0,1]\). Since \(\gamma(\emptyset)=0\), choosing
\(\mathcal Y=\emptyset\) and \(n=j\) gives
\[
B_j:=\varphi^{(j)}(B^j)\in\mathcal F_j .
\]
Moreover, \(B_j\) is compact in \(X\), because it is the continuous image of the
compact set \(B^j\).

Therefore,
\[
c_{\lambda,j}
\le
\sup_{u\in B_j\setminus W}I_\lambda(u)
\le
\sup_{u\in B_j}\widetilde I(u)
=:\beta_j .
\]
Since \(B_j\subset X\) is compact and \(\widetilde I\) is finite and continuous
on \(X\), we have
\[
\beta_j<+\infty .
\]
The set \(B_j\) and the functional \(\widetilde I\) are independent of
\(\lambda\). Hence \(\beta_j\) is independent of \(\lambda\). The proof is
complete.
\end{proof}

We shall use the following elementary consequence of the weak formulation to prove Theorem.

\begin{Lem}\label{Lem3.2}
Let \(u\in H_V^1(\R^N)\) satisfy
\[
\int_{\R^N}\bigl[\nabla u\cdot\nabla\varphi+V(x)u\varphi\bigr]\,dx
=
\int_{\R^N}u\log u^2\,\varphi\,dx,
\qquad \varphi\in C_0^\infty(\R^N).
\]
Then
\[
\int_{\R^N}u^2|\log u^2|\,dx<+\infty,
\]
and
\[
\int_{\R^N}\bigl[|\nabla u|^2+V(x)u^2\bigr]\,dx
=
\int_{\R^N}u^2\log u^2\,dx.
\]
\end{Lem}
\begin{proof}
Let
\[
F(t)=(t^2\log t^2)_+,
\qquad
G(t)=(t^2\log t^2)_- .
\]
For every \(q\in(2,2^*)\), there exists \(C_q>0\) such that
\[
F(t)\le C_q|t|^q,
\qquad t\in\R.
\]
Since \(u\in H^1(\R^N)\subset L^q(\R^N)\), it follows that
\[
F(u)\in L^1(\R^N).
\]

It remains to prove \(G(u)\in L^1(\R^N)\). Let
\[
T_k(t)=\max\{-k,\min\{t,k\}\},
\qquad k>1,
\]
and let \(\eta_R\in C_0^\infty(\R^N)\) satisfy
\[
0\le\eta_R\le1,
\qquad
\eta_R=1 \text{ in }B_R,
\qquad
\eta_R=0 \text{ in }\R^N\setminus B_{2R},
\qquad
|\nabla\eta_R|\le \frac{C}{R}.
\]
The function \(\eta_R^2T_k(u)\) belongs to \(H_0^1(B_{2R})\cap L^\infty(\R^N)\), and it is an admissible test function by the standard truncation and density argument. Testing the equation with \(\eta_R^2T_k(u)\), we obtain
\[
\begin{aligned}
&\int_{\R^N}\eta_R^2T_k'(u)|\nabla u|^2\,dx
+
2\int_{\R^N}\eta_R T_k(u)\nabla u\cdot\nabla\eta_R\,dx \\
&\quad+
\int_{\R^N}V(x)u\,\eta_R^2T_k(u)\,dx
=
\int_{\R^N}\eta_R^2 u\log u^2\,T_k(u)\,dx .
\end{aligned}
\]
For \(t\ne0\), set
\[
\theta_k(t)=\frac{T_k(t)}{t},
\]
and set \(\theta_k(0)=1\). Then \(0\le\theta_k\le1\), \(\theta_k(t)\to1\) as
\(k\to\infty\), and
\[
u\log u^2\,T_k(u)
=
\theta_k(u)u^2\log u^2
=
\theta_k(u)F(u)-\theta_k(u)G(u).
\]
Hence
\[
\begin{aligned}
&\int_{\R^N}\eta_R^2\theta_k(u)G(u)\,dx
+
\int_{\R^N}\eta_R^2T_k'(u)|\nabla u|^2\,dx
+
\int_{\R^N}V(x)u\,\eta_R^2T_k(u)\,dx \\
&=
\int_{\R^N}\eta_R^2\theta_k(u)F(u)\,dx
-
2\int_{\R^N}\eta_R T_k(u)\nabla u\cdot\nabla\eta_R\,dx .
\end{aligned}
\]
Since \(T_k'(u)\ge0\), \(uT_k(u)\ge0\), and \(V(x)\ge V_0>0\), the second and third terms on the left-hand side are nonnegative. Therefore
\[
\int_{\R^N}\eta_R^2\theta_k(u)G(u)\,dx
\le
\int_{\R^N}\eta_R^2\theta_k(u)F(u)\,dx
+
2\int_{\R^N}\eta_R |T_k(u)||\nabla u||\nabla\eta_R|\,dx .
\]
Using \(|T_k(u)|\le |u|\), we get
\[
\int_{\R^N}\eta_R^2\theta_k(u)G(u)\,dx
\le
\int_{\R^N}\eta_R^2F(u)\,dx
+
2\int_{\R^N}\eta_R |u||\nabla u||\nabla\eta_R|\,dx .
\]
Letting \(k\to\infty\) and using the monotone convergence theorem on the left-hand side gives
\[
\int_{\R^N}\eta_R^2G(u)\,dx
\le
\int_{\R^N}\eta_R^2F(u)\,dx
+
2\int_{\R^N}\eta_R |u||\nabla u||\nabla\eta_R|\,dx .
\]
Since \(u\in H^1(\R^N)\), we have
\[
\int_{\R^N}\eta_R |u||\nabla u||\nabla\eta_R|\,dx
\le
\frac{C}{R}|u|_2|\nabla u|_2
\le C
\qquad\text{for }R>1.
\]
Moreover,
\[
\int_{\R^N}\eta_R^2F(u)\,dx\le \int_{\R^N}F(u)\,dx<+\infty.
\]
Thus
\[
\sup_{R>1}\int_{B_R}G(u)\,dx
\le
\sup_{R>1}\int_{\R^N}\eta_R^2G(u)\,dx
<+\infty.
\]
Letting \(R\to+\infty\), we obtain
\[
G(u)\in L^1(\R^N).
\]
Consequently,
\[
u^2|\log u^2|=F(u)+G(u)\in L^1(\R^N).
\]

It remains to prove the energy identity. Since
\[
u^2\log u^2\in L^1(\R^N),
\]
we may use \(\eta_R^2u\) as a test function, again by density. This gives
\[
\int_{\R^N}\eta_R^2|\nabla u|^2\,dx
+
2\int_{\R^N}\eta_R u\nabla u\cdot\nabla\eta_R\,dx
+
\int_{\R^N}V(x)\eta_R^2u^2\,dx
=
\int_{\R^N}\eta_R^2u^2\log u^2\,dx .
\]
The cross term satisfies
\[
\left|
2\int_{\R^N}\eta_R u\nabla u\cdot\nabla\eta_R\,dx
\right|
\le
\frac{C}{R}|u|_2|\nabla u|_2
\to0
\quad\text{as }R\to+\infty.
\]
Since \(u\in H_V^1(\R^N)\) and \(u^2\log u^2\in L^1(\R^N)\), the other terms converge by dominated convergence. Therefore
\[
\int_{\R^N}\bigl[|\nabla u|^2+V(x)u^2\bigr]\,dx
=
\int_{\R^N}u^2\log u^2\,dx .
\]
The proof is complete.
\end{proof}

The following lemma is important to prove Theorem \ref{Thm1.1}.
\begin{Lem}\label{Lem3.3}
Suppose that \(\lambda_n\to0^+\), \(I_{\lambda_n}'(u_n)=0\), and there exist constants \(C>0\) and \(\alpha_0>0\) such that
\[
\alpha_0\le I_{\lambda_n}(u_n)\le C
\quad\text{for all }n.
\]
Then, up to a subsequence,
\[
u_n\to u_0
\quad\text{in }H_V^1(\R^N),
\]
where \(u_0\in H_V^1(\R^N)\), \(u_0\ne0\), and \(u_0\) is a weak solution of \eqref{eq1.1}. Moreover,
\[
I_{\lambda_n}(u_n)\to I(u_0).
\]
\end{Lem}

\begin{proof}
\textbf{Step 1. Uniform bounds.}
Since \(I_{\lambda_n}'(u_n)=0\),
\[
2I_{\lambda_n}(u_n)-\langle I_{\lambda_n}'(u_n),u_n\rangle
=
\frac{2-p}{p}\lambda_n|u_n|_p^p+|u_n|_2^2 .
\]
Thus
\[
\sup_n |u_n|_2<+\infty,
\qquad
\sup_n \lambda_n|u_n|_p^p<+\infty.
\]
Using the logarithmic Sobolev inequality as in Lemma~\ref{Lem2.5}, we obtain
\[
\sup_n\|u_n\|_{H_V^1}<+\infty .
\]
Hence, up to a subsequence,
\[
u_n\rightharpoonup u_0
\quad\text{in }H_V^1(\R^N),
\]
and, by the compact embedding induced by \((V)\),
\[
u_n\to u_0
\quad\text{in }L^q(\R^N),
\quad 2\le q<2^*,
\]
as well as \(u_n(x)\to u_0(x)\) a.e. in \(\R^N\).

\textbf{Step 2. Passage to the limit in the equation.}
Let \(\varphi\in C_0^\infty(\R^N)\). From \(I_{\lambda_n}'(u_n)=0\),
\[
\lambda_n\int_{\R^N}|u_n|^{p-2}u_n\varphi\,dx
+
\int_{\R^N}\bigl[\nabla u_n\cdot\nabla\varphi+V(x)u_n\varphi\bigr]\,dx
-
\int_{\R^N}u_n\log u_n^2\,\varphi\,dx=0.
\]
By Hölder's inequality and \(\lambda_n|u_n|_p^p\le C\),
\[
\left|
\lambda_n\int_{\R^N}|u_n|^{p-2}u_n\varphi\,dx
\right|
\le
C\lambda_n^{1/p}(\lambda_n|u_n|_p^p)^{(p-1)/p}|\varphi|_p
\to0.
\]
The linear term converges by weak convergence in \(H_V^1\).
For the logarithmic term, let \(K=\operatorname{supp}\varphi\) and choose \(q\in(2,2^*)\). With \(r=q/(q-1)\),
\[
|t\log t^2|^r\le C(1+|t|^q),
\quad t\in\R.
\]
Since \(u_n\to u_0\) in \(L^q(K)\), Vitali's theorem gives
\[
u_n\log u_n^2\to u_0\log u_0^2
\quad\text{in }L^r(K).
\]
Letting \(n\to\infty\), we obtain
\[
\int_{\R^N}\bigl[\nabla u_0\cdot\nabla\varphi+V(x)u_0\varphi\bigr]\,dx
=
\int_{\R^N}u_0\log u_0^2\,\varphi\,dx,
\qquad
\varphi\in C_0^\infty(\R^N).
\]
By Lemma~\ref{Lem3.2},
\begin{equation}\label{eq3.1}
\int_{\R^N}\bigl[|\nabla u_0|^2+V(x)u_0^2\bigr]\,dx
=
\int_{\R^N}u_0^2\log u_0^2\,dx.
\end{equation}

\textbf{Step 3. Strong convergence.}
Testing \(I_{\lambda_n}'(u_n)=0\) with \(u_n\), we have
\begin{equation}\label{eq3.2}
\lambda_n|u_n|_p^p+
\int_{\R^N}\bigl[|\nabla u_n|^2+V(x)u_n^2\bigr]\,dx
=
\int_{\R^N}u_n^2\log u_n^2\,dx .
\end{equation}
For \(q\in(2,2^*)\),
\[
(u_n^2\log u_n^2)_+\le C_q|u_n|^q.
\]
Since \(u_n\to u_0\) in \(L^q(\R^N)\), Lemma~\ref{Lem2.3} gives
\[
\int_{\R^N}(u_n^2\log u_n^2)_+\,dx
\to
\int_{\R^N}(u_0^2\log u_0^2)_+\,dx .
\]
Fatou's lemma gives
\[
\int_{\R^N}(u_0^2\log u_0^2)_-\,dx
\le
\liminf_{n\to\infty}
\int_{\R^N}(u_n^2\log u_n^2)_-\,dx.
\]
Consequently,
\[
\limsup_{n\to\infty}
\int_{\R^N}u_n^2\log u_n^2\,dx
\le
\int_{\R^N}u_0^2\log u_0^2\,dx .
\]
Together with \eqref{eq3.1} and \eqref{eq3.2}, this yields
\[
\limsup_{n\to\infty}\|u_n\|_{H_V^1}^2
\le
\|u_0\|_{H_V^1}^2 .
\]
The reverse inequality follows from weak lower semicontinuity. Hence
\[
u_n\to u_0
\quad\text{in }H_V^1(\R^N).
\]
Moreover, \eqref{eq3.1} and \eqref{eq3.2} imply
\[
\lambda_n|u_n|_p^p\to0
\]
and
\[
\int_{\R^N}u_n^2\log u_n^2\,dx
\to
\int_{\R^N}u_0^2\log u_0^2\,dx .
\]

\textbf{Step 4. Nontriviality and energy convergence.}
Since
\[
I_{\lambda_n}(u_n)
=I(u_n)+\frac{\lambda_n}{p}|u_n|_p^p,
\]
the preceding convergences imply
\[
I_{\lambda_n}(u_n)\to I(u_0).
\]
Therefore
\[
I(u_0)=\lim_{n\to\infty}I_{\lambda_n}(u_n)
\ge \alpha_0>0,
\]
and hence \(u_0\ne0\).
\end{proof}

\begin{proof}[{\bf Proof of Theorem~\ref{Thm1.1}}]
Let \(\{\lambda_n\}\subset(0,1]\) be any sequence such that \(\lambda_n\to0^+\). For each \(j\ge j_0\), set
\[
\omega_{n,j}=\omega_{\lambda_n,j}.
\]
By Lemma~\ref{Lem3.1},
\[
\alpha_j\le I_{\lambda_n}(\omega_{n,j})\le\beta_j,
\qquad j\ge j_0,
\]
where \(\alpha_j,\beta_j\) are independent of \(n\), and \(\alpha_j\to+\infty\).

Fix \(j\ge j_0\). Applying Lemma~\ref{Lem3.3} to \(u_n=\omega_{n,j}\), and passing to a subsequence depending on \(j\), we obtain
\[
\omega_{n,j}\to\omega_j
\quad\text{in }H_V^1(\R^N),
\]
where \(\omega_j\ne0\) is a weak solution of \eqref{eq1.1}. Moreover,
\[
I(\omega_j)
=
\lim_{n\to\infty}I_{\lambda_n}(\omega_{n,j})
\ge
\alpha_j.
\]
Since \(I(\omega_j)\ge\alpha_j\) and \(\alpha_j\to+\infty\), the set \(\{\omega_j:j\ge j_0\}\) contains infinitely many distinct elements.

It remains to prove that each \(\omega_j\) is sign-changing. Since \(\omega_{n,j}\) is a sign-changing critical point of \(I_{\lambda_n}\), both \((\omega_{n,j})_+\) and \((\omega_{n,j})_-\) are nontrivial. Testing \(I_{\lambda_n}'(\omega_{n,j})=0\) with \((\omega_{n,j})_+\), we get
\[
\lambda_n| (\omega_{n,j})_+|_p^p+
\int_{\R^N}\bigl[|\nabla(\omega_{n,j})_+|^2+V(x)(\omega_{n,j})_+^2\bigr]\,dx
=
\int_{\R^N}(\omega_{n,j})_+^2\log (\omega_{n,j})_+^2\,dx.
\]
Thus, by Sobolev's inequality and \((t^2\log t^2)_+\le C|t|^{2^*}\),
\[
S|(\omega_{n,j})_+|_{2^*}^2
\le
C|(\omega_{n,j})_+|_{2^*}^{2^*}.
\]
Since \((\omega_{n,j})_+\not\equiv0\), it follows that
\[
|(\omega_{n,j})_+|_{2^*}\ge c_0>0,
\]
with \(c_0\) independent of \(n\). The strong convergence in \(H_V^1(\R^N)\) implies
\[
(\omega_{n,j})_+\to(\omega_j)_+
\quad\text{in }L^{2^*}(\R^N),
\]
and hence \((\omega_j)_+\not\equiv0\).

Testing with \((\omega_{n,j})_-\) and multiplying the resulting identity by \(-1\), the same argument gives
\[
|(\omega_{n,j})_-|_{2^*}\ge c_0>0.
\]
Passing to the limit yields \((\omega_j)_-\not\equiv0\). Therefore \(\omega_j\) is sign-changing. After relabeling, the sequence \(\{\omega_j\}_{j\ge j_0}\) gives infinitely many sign-changing weak solutions of \eqref{eq1.1}.
\end{proof}

\section*{Acknowledgment}
The authors would like to express sincere gratitude to his supervisor Prof. Wang Zhi-Qiang for his valuable advise. 

C. Huang was supported by China Postdoctoral Science Foundation (2020M682065). Z. Yang was supported by National Natural Science Foundation of China (12301145, 12261107, 12561020) and Yunnan Fundamental Research Projects (202401AU070123, 202601AT070048). J. Zhou was supported by CNPq/Brazil (Grant No. 304627/2023-2).

\medskip
{\bf Author Contributions:} All authors contributed equally to the writing and preparation of the manuscript.

\medskip
{\bf Data availability:}  Data sharing is not applicable to this article as no new data were created or analyzed in this study.

\medskip
{\bf Conflict of Interests:} The authors declare that they have no conflict of interest.

\end{document}